\documentclass[english,11pt,letterpaper]{article}
\usepackage[margin=1.0in]{geometry}
\usepackage[numbers]{natbib}

\usepackage{algorithm}
\usepackage{algorithmic}
\usepackage{csquotes} 
\usepackage{multirow}
\usepackage[pdftex,colorlinks]{hyperref}
\usepackage{float}
\usepackage{amssymb}
\usepackage{amsmath}
\usepackage{amsfonts}
\usepackage{amssymb}
\usepackage{amsthm}
\usepackage{algorithm}
\usepackage{algorithmic}
\newtheorem{theorem}{Theorem}[section]

\newtheorem{counter-example}[theorem]{Counter example}

\newtheorem{open question}[theorem]{Open question}
\newtheorem{corollary}[theorem]{Corollary}

\newtheorem{claim}[theorem]{Claim}

\newcommand{\eps}{{\varepsilon}}

\newcommand{\mdnote}[1]{{\color{red} \bf (#1 --Mike)}}

\makeatletter

\floatstyle{ruled}
\newfloat{algorithm}{tbp}{loa}
\providecommand{\algorithmname}{Algorithm}
\floatname{algorithm}{\protect\algorithmname}

\makeatother

\usepackage{babel}
\begin{document}

\begin{titlepage}
\author{Michael Dinitz\thanks{Dept.~of Computer Science, Johns Hopkins University.  Email: \texttt{mdinitz@cs.jhu.edu}}   \hspace{1cm} Michael Schapira\thanks{School of Computer Science and Engineering, The Hebrew University, Jerusalem, Israel. Email: \texttt{schapiram@cs.huji.ac.il}} \hspace{1cm} Gal Shahaf\thanks{Dept. of Mathematics, The Hebrew University, Jerusalem, Israel. Email: \texttt{gal.shahaf@mail.huji.ac.il}}
}

\title{Large Low-Diameter Graphs are Good Expanders}
\date{}

\maketitle

\begin{abstract}
We revisit the classical question of the relationship between the diameter of a graph and its expansion properties. One direction is well understood: expander graphs exhibit essentially the lowest possible diameter. We focus on the reverse direction, showing that ``sufficiently large'' graphs of fixed diameter and degree must be ``good'' expanders. We prove this statement for various definitions of ``sufficiently large'' (multiplicative/additive factor from the largest possible size), for different forms of expansion (edge, vertex, and spectral expansion), and for both directed and undirected graphs. A recurring theme is that the lower the diameter of the graph and (more importantly) the larger its size, the better the expansion guarantees. Aside from inherent theoretical interest, our motivation stems from the domain of network design. Both low-diameter networks and expanders are prominent approaches to designing high-performance networks in parallel computing, HPC, datacenter networking, and beyond. Our results establish that these two approaches are, in fact, inextricably intertwined. We leave the reader with many intriguing questions for future research.
\end{abstract}

\end{titlepage}

\section{Introduction}

Both the diameter of a graph and its expansion capture the ``connectedness'' of the graph, albeit in two very different senses. The diameter, i.e., the maximal distance between a pair of vertices, provides an upper bound on the length of shortest paths in the graph, whereas expansion measures the minimal ratio between a subset of vertices and its boundary. We revisit the classical question of relating these two traits. One direction is well known: good expansion implies a low diameter. Specifically, the diameter of a graph with good expansion is $O(\log n)$ (see, e.g.,~\cite{hoory2006expander}), which is asymptotically the lowest possible. We focus on the opposite, and largely unexplored, direction.

In general, low diameter does not guarantee good expansion. Consider, e.g., a graph on $n$ vertices that is a disjoint union of two cliques, each of size $\frac{n}{2}$. Removing one edge from each clique and connecting the cliques via two ``bridges'' results in a $(\frac{n}{2}-1)$-regular graph of diameter $3$ with very low expansion (which worsens as $n\rightarrow \infty$). We observe, however, that this ``bad'' graph is significantly smaller than the largest $(\frac{n}{2}-1)$-regular graph of diameter $3$ (which is of size $\Omega(n^3)$). Indeed, our investigation below reveals that, in contrast to the above, when the degree and the diameter are fixed and the size of the graph is ``sufficiently large'', the graph must have ``good'' expansion. We formalize this statement for different notions of ``large'', for different forms of expansion (edge, vertex, and spectral expansion), and for undirected/directed graphs.  Our results are presented in Section~\ref{sec:results}, but informally, ``sufficiently large'' means that the size of the graph is close to the best-known upper bound on the size, in either a multiplicative or additive setting.

We formalize the above statements and discuss implications of our results for network design and beyond, including the unification of two competing approaches to datacenter network design.

%The best-known upper bound to this problem is known in the literature as the \textit{Moore Bound}, denoted by $\mu_{d,k}$ and derived by summation of the maximal number of vertices at distance at most $k$ from a fixed root. Over the years, extensive research has been carried out in order to determine the existence of graphs that achieve this upper bound (a.k.a.~Moore Graphs), or at least come close to it. This line of study, known as the \textit{``degree/diameter problem"}, was initiated by Hoffman and Singleton~\cite{hoffman1960moore} with their celebrated result that a Moore Graph of diameter 2 must have degree 2, 3, 7, or 57. Subsequent studies have shown that no Moore Graphs exist whenever $k\geq 3$~\cite{damerell1973moore}, and that for every $d$ and $k$ only finitely many $(d,k)$-graphs are of size $\mu_{d,k}-1$ \cite{bannai1981regular}. For a comprehensive survey of related results in this active field of research, we refer the reader to~\cite{miller2005moore}. For our purposes, it is enough to note that the Moore Bound is still essentially the state of the art upper bound for the size of $(d,k)$-graphs, and hence we use this bound as a benchmark and compare the size of $(d,k)$-graphs to  $\mu_{d,k}$. In particular, we measure the size of the graph in terms of approximation, either additive or multiplicative, to the Moore Bound.\\

\subsection{The Degree/Diameter Problem} 

Before we can state our results, we must first define what we mean by a ``large, low-diameter graph".  To start off: how large can a $d$-regular graph of diameter $k$ (which we shall refer to as a ``$(d,k)$-graph'' henceforth) actually be? An upper bound on the size of such a graph is the classical \textit{Moore Bound}~\cite{hoffman1960moore}, denoted by $\mu_{d,k}$ (see Section~\ref{Preliminaries} for a formal definition). Extensive research has been devoted to determining the existence of graphs whose sizes match this upper bound (a.k.a.,~Moore Graphs) or well-approximate it. This line of study, termed the \textit{``degree/diameter problem''}, was initiated by Hoffman and Singleton~\cite{hoffman1960moore}. See~\cite{miller2005moore} for a detailed survey of results in this active field of research.

Graphs whose sizes exactly match the Moore Bound, referred to as ``Moore Graphs'' henceforth, only exist for very few values of $d$ and $k$~\cite{hoffman1960moore,bannai1981regular}. Consequently, various constructions for generating graphs whose sizes come ``close'' to the Moore Bound, which we call ``approximate-Moore Graphs'', have been devised.

Specifically, constant multiplicative approximations (MMS-graphs~\cite{mckay1998note}) and constant additive approximations (e.g., polarity graphs~\cite{erdos1962problem,brown1966graphs}) have been devised for the case of diameter $k=2$. Delorme~\cite{miller2005moore, delorme1985grands, delorme1985large} constructed infinite series of $(d,k)$-graphs whose sizes arbitrarily approach the Moore Bound for diameters $k=3$ and $k=5$. Graph constructions whose sizes approximate the Moore Bound within non-constant multiplicative factors exist for arbitrary values of $k$ (examples include, e.g., \textit{de Bruijn} \cite{de1946combinatorial} and \textit{Canale-Gomez} \cite{canale2005asymptotically} graphs for the undirected case, and \textit{Alegre} and \cite{miller2005moore}
\textit{Kautz} \cite{elspas1968theory} digraphs for the directed analogue of the problem). While constructing approximate-Moore Graphs whose sizes arbitrarily approach the Moore Bound for arbitrary values of $k$ remains an important and widely studied open question, such graphs are believed to exist for sufficiently large $d$ and $k$, as conjectured, e.g., by Bollob{\'a}s in~\cite{bollobas1978extremal}.

Our investigation of the relation between diameter and expansion also uses the Moore Bound as a benchmark and compares the size of $(d,k)$-graphs to $\mu_{d,k}$. We consider both multiplicative and additive approximations to the Moore Bound. Our results establish that good solutions to the degree-diameter problem \emph{must} be good expanders, establishing a novel link between two prominent and classical lines of research. In addition, our results yield new expansion bounds for all of the classical constructions of low-diameter graphs discussed above.

%We point out that proving Bollob{\'a}s' conjecture (or even extending Delorme's results for other specific values of $k$ and $d$) would immediately imply, by our bounds, that the resulting graphs are good expanders. The reader is referred to Section~\ref{spectral} and Appendix~\ref{sec: implications} for more details. 

\subsection{Our Results} \label{sec:results} 

Our results relating the size of a $(d,k)$-graph to its expansion are summarized in Table~\ref{results}, with $\lambda(G)$, $h_e(G)$, and $\phi_V(G)$ denoting spectral expansion, edge expansion, and vertex expansion, respectively (formal definitions can be found in Section~\ref{Preliminaries}). $\widetilde{\mu_{d,k}}$ is the analogue of the Moore Bound for directed graphs.

\begin{table}[]\label{results}
	\centering
	\label{my-label}
	\begin{tabular}{|c|c|c|}
		\hline
		& \textbf{Size}                                           & \textbf{Expansion guarantees}                                           \\ \hline
		\multirow{4}{*}{\textbf{$(d,k)$-graph}}   & $n\geq \mu_{d,k} - O(d^{k/2})$                          & $\lambda(G) = O(\sqrt{d})$                                              \\ \cline{2-3} 
        & $n \geq (1-\eps) \mu_{d,k}$ & $\lambda(G) = O(\eps^{1/k})d$ \\ \cline{2-3}
		& \multirow{2}{*}{$n=\alpha \cdot \mu_{d,k}$}             & $h_e(G)\geq \frac{\alpha d}{2k} \cdot \left(1-\frac{1}{(d-1)^k}\right)$ \\ \cline{3-3} 
		&                                                         & $\phi_V(G)\geq \frac{\alpha}{2(k-1) +\alpha}$                           \\ \hline
		\multirow{3}{*}{\textbf{$k=2$}}           & $n$                                                     & $\lambda(G)\leq \frac{1+ \sqrt{1+4(d^2+d-n)}}{2}$                       \\ \cline{2-3} 
		& \multirow{2}{*}{$n=\alpha \cdot d^2$}                   & $h_e(G)\geq \frac{2d+1-\sqrt{4(1-\alpha)d^2+4d+1}}{4}$                  \\ \cline{3-3} 
		&                                                         & $\phi_V(G)\geq \frac{2\alpha}{2\alpha+1}$                               \\ \hline
		\textbf{$k=3$}                            & $n=\alpha \cdot d^3$                                    & $\phi_V(G)\geq \frac{\alpha}{\alpha+1}$                                 \\ \hline
		\multirow{2}{*}{\textbf{$(d,k)$-digraph}} & \multirow{2}{*}{$n=\alpha \cdot \widetilde{\mu_{d,k}}$} & $h_e(G)\geq \frac{\alpha}{2k} (d-\frac{1}{d^k})$                        \\ \cline{3-3} 
		&                                                         & $\phi_V(G)\geq \frac{\alpha \cdot d}{2(d+1)(k-1)+\alpha \cdot d}$       \\ \hline
	\end{tabular}
	\caption{Summary of Results Relating Expansion and Diameter.}
\end{table}

We begin in Section~\ref{spectral} with our main results, which provide bounds on the spectral expansion of large $(d,k)$-graphs.  In particular, our results establish that if the size of a graph is very close \emph{additively} to the Moore bound, then the graph is essentially an optimal expander. In addition, if the graph has size that is close \emph{multiplicatively} to the Moore bound, the spectral expansion might no longer be optimal, but is still very good.

We next turn our attention to combinatorial notions of expansion: edge expansion and vertex expansion. We provide (in Section~\ref{coarse}) guarantees on both the edge and the vertex expansion of $(d,k)$-graphs in terms of their multiplicative distance from the Moore Bound. Our analysis leverages careful counting arguments to bound the ratio between the cardinality of a set of vertices and the size of its boundary. We also prove, through more refined analyses, improved results for diameters $2$ and $3$. 

\noindent{\bf Our technique.} The key technical insight underlying our results for spectral expansion is a novel link, which we believe is of independent interest, between the nontrivial eigenvalues of a graph's adjacency matrix and the distance of the graph from the Moore Bound. Specifically, the proofs of our results for spectral expansion rely on the analysis of \emph{non-backtracking paths} in the graph. A path is said to be non-backtracking if it does not traverse an edge back and forth consecutively. We prove that the matrix that corresponds to all non-backtracking paths of length at most $k$ must consist of strictly positive entries and shares all eigenvectors of the adjacency matrix $A$. We establish the above algebraic relation between the two matrices by employing the \textit{Geronimus Polynomials}~\cite{biggs1993algebraic,sole1992second}, a well-known class of orthogonal polynomials, as operators acting on the adjacency matrix. Given the spectrum of $A$, an asymptotic estimation of the polynomials' coefficients allows us to bound the spectrum of the non-backtracking paths matrix. We then subtract from the latter the all-ones matrix and use the leading eigenvalue of the remaining matrix (which can be computed directly) to bound the nontrivial eigenvalues of the adjacency matrix $A$.

Our technique should be contrasted with employing Hashimoto's non-backtracking operator~\cite{hashimoto1989zeta} to reason about non-backtracking paths in a graph (e.g., in the context of localization and centrality~\cite{martin2014localization}, clustering~\cite{krzakala2013spectral}, mixing time acceleration \cite{alon2007non}, and percolation~\cite{karrer2014percolation} in networks). In our context, applying Hashimoto's operator involves reasoning about intricate relations between the spectra of the adjacency matrix $A$ and another matrix, called the ``non-backtracking matrix'' (via the Ihara-Bass formula~\cite{bordenavenew}). Geronimous Polynomials, a subfamily of the more renowned Chebyshev polynomials, allow for much simpler analysis. We believe that our techniques, and the (yet to be explored) connections between Geronimous Polynomials and Hashimoto's operator, are of independent interest and may find wider applicability.

%Our main results (in Section~\ref{spectral}) establishes that if $G$ is a $(d,k)$-graph of size $n\geq \mu_{d,k} - O(d^{k/2})$, then $\lambda(G) = O(\sqrt{d})$, where $\lambda(G)$ is the second largest absolute eigenvalue of $G$. Since $\lambda(G)$ is known to be at least $\sqrt{d}$~\cite{hoory2006expander}, the spectral gap of such graphs is essentially optimal.

%\bigskip

\noindent{\bf ``Not-so-sparse'' expanders.} Importantly, our research diverges from the main vein of prior research on expanders. Expanders are commonly viewed as highly-connected sparse graphs. Indeed, the bulk of literature on this topic assumes that the degree of these graphs is essentially constant with respect to the size of the graph (i.e., $d \ll n$). In contrast, the size of a ``large'' $(d,k)$-graph graph is $O(d^k)$.

\begin{table}[]
\centering
%\label{my-label}
\begin{tabular}{|c|c|c|c|}
\hline
\textbf{Construction}   & \textbf{Algebraic expansion} & \textbf{Edge expansion}     & \textbf{Vertex expansion}                                                                                                           \\ \hline
\textit{de Bruijn} (with $k=2$) \cite{de1946combinatorial}    &                           -                        & $-$                                   & $\frac{1}{3}$                                                                                                     \\ \hline
\textit{Canale-Gomez} \cite{canale2005asymptotically}     &                              -         & $\frac{d}{2k\cdot 1.57^k} \cdot \left(1-\frac{1}{(d-1)^k}\right)$                                & $\frac{1.57^{-k}}{2(k-1) +1.57^{-k}}$                                                                                               \\ \hline
\textit{Alegre digraph} \cite{miller2005moore}                 &                -                  & $\frac{25\cdot 2^k}{32k\cdot d^k} \left(d-\frac{1}{d^k}\right)$ & $\frac{\left(\frac{2}{d}\right)^k \cdot \frac{25}{16} \cdot d}{2(d+1)(k-1)+\left(\frac{2}{d}\right)^k \cdot \frac{25}{16} \cdot d}$ \\ \hline
\textit{Kautz digraphs} \cite{elspas1968theory}                &         -              & $\frac{1}{2k} \left(d-\frac{1}{d^k}\right)$                                                      & $\frac{d}{2(d+1)(k-1)+d}$                                                                                                           \\ \hline
\textit{Polarity graph} \cite{erdos1962problem,brown1966graphs} & $\lambda(G)\leq \frac{1+\sqrt{1+8(d-1)}}{2}$ & $\frac{2d+1-\sqrt{4d+1}}{4}$                                                                     & $\frac{2}{3}$                                                                                                                       \\ \hline
\textit{MMS-graphs} \cite{mckay1998note}                                  &          $\lambda(G)\leq \frac{1+\frac{1}{3}\sqrt{d^2+d+7}}{2}$             & $\frac{2d+1-\sqrt{\frac{4}{9}d^2+4d+1}}{4}$                                  & $\frac{16}{25}$                                                                                                                     \\ \hline
\end{tabular}
\caption{Implications for Known Constructions of Low-Diameter Graphs}
\label{implications table}
\end{table}

\noindent{\bf Implications for network design and beyond.} Aside from inherent theoretical interest, our motivation stems from the domain of network design. Low-diameter networks have been widely studied in the context of high-performance-computing (HPC) architectures (see, e.g.,~\cite{kim2008technology,besta2014slim,arimilli2010percs,kim2007flattened}), parallel computing~\cite{leighton2014introduction}, and the design of fault-tolerant networks~\cite{bermond1989large,bermond1989bruijn,besta2014slim,esfahanian1985fault,pradhan1985fault}. Of special interest in this literature are large networks of very low diameters (e.g., $2$ or $3$), as short path lengths translate to low latency in data delivery and also to low packet-queueing delays and power consumption (due to having few intermediate network devices en route to traffic destinations~\cite{erdos1962problem,brown1966graphs,mckay1998note,besta2014slim,kim2008technology,kim2009cost}). Similarly to low-degree networks, expanders have been shown to induce high performance in a broad spectrum of network design contexts.

Recently, the focus on either the diameter or the expansion of a network topology gave rise to two competing approaches for datacenter architecture design~\cite{valadarsky2016xpander,singla2014high,lucian2011jellyfish,besta2014slim,kim2008technology,kim2009cost}. Specifically, an important line of research in datacenter design (see, e.g.,~\cite{singla2014high,lucian2011jellyfish,besta2014slim,guo13}) relies on (either implicitly or explicitly) utilizing graphs whose sizes are as large as possible for a given diameter and degree as datacenter network topologies\footnote{The authors of~\cite{lucian2011jellyfish}, for instance, write that ``Intuitively, the best known degree-diameter topologies should support a large number of servers with high network bandwidth and low cost (small degree)... Thus, we propose the best-known degree-diameter graphs as a benchmark for comparison.''}. A different strand of research investigates how utilizing expander graphs as datacenter network topologies can be turned into an operational reality~\cite{valadarsky2016xpander,dinitz2016xxx,kassing2017static}. 

Our results show that these two approaches are, in fact, inextricably intertwined; not only do expanders exhibit low (in fact, near-optimal) diameters~\cite{hoory2006expander}, but constructing large low-diameter datacenter networks effectively translates to constructing good expanders.  Thus these two approaches to designing datacenter networks can essentially be regarded as one: the search for extremely strong expanders. Our results provide new expansion guarantees for a number of well-studied low-diameter networks, including MMS graphs~\cite{mckay1998note} (proposed for the context high-performance computing and datacenters, see \textit{Slim Fly}~\cite{besta2014slim}), polarity graphs~\cite{erdos1962problem,brown1966graphs}, Canale-Gomez graphs~\cite{canale2005asymptotically}, and more. See summary of the implications of our results for different graph constructions in Table~\ref{implications table}. Our results for spectral expansion essentially match previously established results, thus generalizing and unifying prior construction-specific bounds.

Beyond the implications for network design, the study of low-diameter networks also pertains to other areas such as feedback registers~\cite{fredricksen1982survey,lempel1970homomorphism} and decoders~\cite{collins1992vlsi}.

\section{Preliminaries} \label{Preliminaries}

We provide below a brief exposition of graph expansion and the Moore Bound. We refer the reader to~\cite{hoory2006expander} and~\cite{miller2005moore} for detailed expositions of these topics.

\noindent{\bf Graph expansion.} Let $G=(V,E)$ be an undirected graph of size $|V|=n$. $G$ is said to be \emph{$d$-regular} if each of its vertices is of degree $d$, and of \emph{diameter} $k$ if the maximum distance between any two vertices in the graph is $k$. $d$-regular graphs of diameter $k$ are denoted throughout the paper as \emph{$(d,k)$-graphs}.

The combinatorial expansion of the graph reflects an isoperimetric view and is the minimal ratio between the boundary $\partial S$ of a set $S$ and its cardinality. Different interpretations of $\partial S$ give rise to different notions of expansion.

The \emph{edge expansion} of $G$ is 
$$h_e(G):=\min_{|S|\leq \frac{n}{2}} \frac{|e(S,S^c)|}{|S|}$$
where $e(S,S^c):=\{(u,v)\in E|u\in S,v\in S^c \}$. 

The \emph{vertex expansion} of $G$ is 
$$\phi_V(G) = \min_{0<|S|\leq \frac{n}{2}} \frac{|N(S)|}{|S|}.$$
where $N(S):=\{v\in S^c|\quad \exists u\in S \ \text{s.t. } (u,v)\in E\}$.\footnote{The definitions of edge and vertex expansion admit several variants, based on either the size of the cut or the type of the boundary (see \cite{hoory2006expander} for examples). While we adopt the most common of those, our results can be stated w.r.t.~other variants as well.}

%The relationship between both definitions of is rather strong: it is easy to verify that
%\[\phi_V(G)\leq h_e(G)\leq d \cdot\phi_V(G).\]

We next define the algebraic (spectral) notion of expansion. Let $A$ be the adjacency matrix of the graph. Since $A$ is symmetric it is diagonalizable with respect to an orthonormal basis, and the corresponding eigenvalues are real, and so can be ordered as follows:
$$\lambda_1\geq\lambda_2\geq...\geq\lambda_n.$$

The first eigenvalue of a $d$-regular graph satisfies $\lambda_1=d$ and has the all-ones vector \textbf{$1_n$} as the associated eigenvector. Let $\lambda(G):=\max\{|\lambda_2|,|\lambda_n|\}$. A graph $G$ is said to be an expander if $\lambda(G)$ is bounded away from $d$ by some constant \cite{arora2009computational}. The \emph{algebraic expansion} (or spectral expansion) is then defined as $d-\lambda(G)$, termed the \emph{spectral gap}. The larger the gap, the better the expansion.

%The latter is of substantial importance as it may be used to bound the mixing time of a random walk on the graph. The relation between the algebraic and the combinatorial expansion is given by the Cheeger inequality:

%\[\frac{d-\lambda_2}{2}\leq h_e(G) \leq \sqrt{2d(d-\lambda_2)}\]

\noindent{\bf The Moore Bound.} How large can a $(d,k)$-graph be? A straightforward upper bound is obtained by summation of the vertices according to their distance from a fixed vertex $v_0\in V$. Let $m_j$ denote the number of vertices at distance $j$ from $v_0$. Note that $m_0 = 1$ and $m_1 = d$. As vertices at distance $j\geq 2$ must be adjacent to some vertex at distance $j-1$, we have that $m_{j}\leq (d-1)m_{j-1}$. A simple induction implies that $m_{j}\leq d(d-1)^{j-1}$. Now since the diameter is $k$, all vertices have distance at most $k$ from $v_0$, and hence $n \leq 1+d+d(d-1) + d(d-1)^2 + ...+ d(d-1)^{k-1}$. We denote this expression, known as the \emph{Moore Bound} of the graph, by

\[\mu_{d,k} := 
1 + d\sum_{i=0}^{k-1}(d-1)^i =
\left\{\begin{matrix} 
2k+1 & \mbox{if } d=2  \\ 
1+d\cdot \frac{(d-1)^k-1}{d-2} & \mbox{if } d>2 \end{matrix}\right.\]

\section{Diameter vs.~Algebraic Expansion} \label{spectral}

We establish below a relationship between the nontrivial eigenvalues of $A$ and the distance of the graph from the Moore Bound. This relationship will enable us to prove a variety of bounds on the algebraic expansion of approximate-Moore graphs. This novel link relies on the following class of orthogonal polynomials: let $P_0(x)=1$, $P_1(x)=x$, $P_2(x) = x^2 - d$, and for every $t>2$ define $P_t(x)$ by the recurrence relation
\[P_t(x) = x P_{t-1}(x) - (d-1)P_{t-2}(x).\] 

The significance of this class of polynomials, termed the \emph{``Geronimus Polynomials''}~\cite{biggs1993algebraic,sole1992second}, is reflected in the main technical theorem of this section:

\begin{theorem} \label{eq: eigenvalue bound}	
	Let $G$ be $(d,k)$-graph of size $n$. Then, every nontrivial eigenvalue $\lambda<d$ of $G$ satisfies
	$$\left|\sum_{t=0}^{k} P_t(\lambda)\right| \leq \mu_{d,k} - n$$
\end{theorem}

Before delving into the proof of Theorem~\ref{eq: eigenvalue bound}, we discuss some of its implications. Theorem~\ref{eq: eigenvalue bound} can be applied to provide meaningful guarantees regarding the spectral expansion of low-diameter graphs whose sizes approach the Moore Bound. Constructing large graphs of very low diameter, e.g., $k=2,3$, has received much attention from both a theoretical perspective (see, e.g.,~\cite{erdos1962problem,brown1966graphs,mckay1998note}) and a practical perspective (see, e.g.,~\cite{besta2014slim,kim2008technology,kim2009cost}). An immediate implication of Theorem~\ref{eq: eigenvalue bound} is the following:

\begin{theorem}\label{thm: spectral k=2}
	Let $G$ be a $d$-regular graph of diameter $k=2$ and size $n$, then 
	$$\lambda(G)\leq \frac{1+ \sqrt{1+4(d^2+d-n)}}{2}.$$
\end{theorem} \label{spec 2-diam thm}
\begin{proof} 
	Applying the Geronimus Polynomials $P_t(\lambda)$ for $0\leq t\leq 2$ in Theorem \ref{eq: eigenvalue bound} yields
	$$\lvert1+ \lambda+(\lambda^2-d)\rvert \leq \mu_{d,2} - n = d^2+1-n.$$
	The result follows from solving the quadratic inequality.
\end{proof}

This theorem immediately bounds the algebraic expansion of polarity graphs~\cite{erdos1962problem,brown1966graphs} and MMS graphs~\cite{mckay1998note} claimed in Table~\ref{implications table}, as both of these classes of graphs have diameter $2$. What about graphs of diameter $k>2$? A more careful analysis of the Geronimus polynomials for larger values of $k$ allows us to use Theorem~\ref{eq: eigenvalue bound} to prove two different expansion bounds. The first is an extremely strong expansion bound but requires the size of the graph to be additively close to the Moore bound, whereas the second allows a small multiplicative gap between the size and the Moore bound but establishes a weaker expansion guarantee.

\begin{theorem}\label{spectral main}
	Let $G$ be a $(d,k)$-graph of size $n\geq\mu_{d,k}-O(d^{k/2})$, for some constant $k>0$. Then
	$\lambda(G)=O(\sqrt{d})$.
\end{theorem} 

Since any $d$-regular graph must satisfy $\lambda(G) \geq \sqrt{d}$ (see \cite{hoory2006expander} for details), Theorem~\ref{spectral main} implies that an additive approximation of $O(d^{k/2})$ of the Moore Bound implies essentially optimal spectral properties. 

%Recall that the celebrated Alon-Boppana bound~\cite{nilli1991second,friedman1991some} does not hold in this scenario (since $n=O(d^k)$, the degree $d$ is not a constant w.r.t.~the size of the graph). Yet, it is still possible to show (via the trace method) that a $d$-regular graph must satisfy $\lambda(G) \geq \sqrt{d}$ (see \cite{hoory2006expander} for details). Hence, Theorem~\ref{spectral main} implies that an additive approximation of $O(d^{k/2})$ of the Moore Bound implies essentially optimal spectral properties. 

\begin{theorem} \label{weak-expansion}
Let $G$ be a $(d,k)$-graph of size $n \geq (1-\eps) \mu_{d,k}$, for some constant $k>0$. Then $\lambda(G) \leq O(\eps^{1/k}) \cdot d$.
\end{theorem}

Delorme~\cite{miller2005moore, delorme1985grands, delorme1985large} proved the existence of an infinite series of $(d,k)$-graphs whose sizes arbitrarily approach the Moore Bound for diameters $k=3$ and $k=5$. Specifically, Delorme proved that $\liminf_{d\rightarrow\infty} \frac{n_{d,k}}{d^k} = 1,$ for $k=3,5$, where $n_{d,k}$ is the largest possible size of a $(d,k)$-graph. This means that for $k=3,5$, for any constant $\eps > 0$  there is some value $d'$ such that for all $d \geq d'$ there is a $(d,k)$-graph with at least $(1-\eps) \mu_{d,k}$ vertices. Hence, Theorem~\ref{weak-expansion} implies that these graphs are good expanders. 

Bollob{\'a}s conjectured that $(d,k)$-graphs of size $n \geq (1-\eps) d^k$ \emph{always exist} for sufficiently large $d$ and $k$~\cite{bollobas1978extremal}. Delorme's results may be perceived as supporting this conjecture. We point out that proving Bollob{\'a}s' conjecture (or even extending Delorme's results to other specific values of $k$ and $d$) would immediately imply, by Theorem~\ref{weak-expansion}, similar expansion guarantees.
%Applying Theorem~\ref{eq: eigenvalue bound} to the context of diameter $k=3,5$ yields bounds on the spectral gap formulated as degree-$k$ polynomial inequalities. The expansion guarantees from Theorem \ref{eq: eigenvalue bound} for graphs of higher diameter can be captures via the following asymptotic formulation.
The remainder of this section is devoted to the proofs of Theorems~\ref{eq: eigenvalue bound},~\ref{spectral main}, and~\ref{weak-expansion}.

%\bigskip

%The outline of the proof is as follows. Let $A$ denote the adjacency matrix of the graph. We begin with the observation that in a $k$-diameter graph, the sum $\sum_{t=0}^{k}A^t$ must consist of strictly positive entries. We than apply a class of polynomials $P_t(x)$ as operators over $A$ to obtain a sparser matrix $\sum_{t=0}^{k} P_t(A)$ that admits strictly positive entries as well. This allows us to subtract from this sum the all-ones matrix $J$ in order to obtain a non-negative matrix $M$. The next step utilizes the fact that $A$ and $M$ share the same eigenvectors, and the corresponding eigenvalues can be computed directly using the polynomials $P_t(x)$. An assymptotic estimation of the coefficients of $P_t$ then implies that if $\lambda(G)$ is of order larger than $\sqrt{d}$, the second eigenvalue of $M$ must be larger then $O(d^{k/2})$. Applying the Perron-Frobenius Theorem~\cite{x}, $O(d^{k/2})$ is shown to be an upper bound on the leading eigenvalue of $M$, thus leading to a contradiction.

%\noindent{\bf Bounding the nontrivial eigenvalues (proof of Theorem~\ref{eq: eigenvalue bound}).} 

\subsection{Bounding the nontrivial eigenvalues (proof of Theorem~\ref{eq: eigenvalue bound})}

Our high-level approach to proving Theorem~\ref{eq: eigenvalue bound} is the following: We aim to bound $\lambda(G)$, the second-largest eigenvalue (in absolute value) of the adjacency matrix $A$. We instead consider a different matrix $M$, obtained by employing the \emph{Geronimus Polynomials} as operators over $A$. The combinatorial properties of this class of polynomials allow us to show that $M\textbf{1}_n = (\mu_{d,k} - n)\textbf{1}_n$. Applying the Perron-Frobenius Theorem asserts that this eigenvalue serves as a bound over the entire spectrum of $M$. We then utilize the algebraic relation between both matrices: Namely, we bound $A$'s nontrivial spectrum, using the fact that $M$ shares the same eigenvectors as $A$, and that its eigenvalues may be derived from those of $A$ via an operation of the Geronimus Polynomials. This will then imply Theorem~\ref{eq: eigenvalue bound}.

We begin with the known solution to the recurrence, formulated via a trigonometric expression that holds for all $t>0$~\cite{szeg1939orthogonal}: 

\begin{equation} \label{eq:geronimus}
P_t(2\sqrt{d-1} \cos \theta) = (d-1)^{t/2-1} \frac{(d-1)\sin ((t+1)\theta) - \sin ((t-1)\theta)}{\sin \theta}
\end{equation}

One can easily check that this identity applies for $t=1,2$ and verify that the recurrence relation holds for $t>2$. All roots of $P_t$ are real and lie in the interval $[-2\sqrt{d-1}, 2\sqrt{d-1}]$~\cite{bartal2003metric,linial2002girth}.

Our framework applies the Geronimus Polynomials as operators over the adjacency matrix $A$. This method has several advantages: Algebraically, since $P_t(A)$ is a linear combination of powers of $A$, each eigenvector $v$ of $A$ is an eigenvector of $P_t(A)$ as well. Thus, the spectrum of $P_t(A)$ is given by $spec[P_t(A)] = \{P_t(\lambda)\> |\> \> \lambda \text{ is an eigenvalue of } A \}$. Viewed from a combinatorial perspective, this operation allows us to dismiss backtracking paths from consideration. By \emph{backtracking}, we refer to paths that traverse an edge in both directions consecutively. Note that a non-backtracking path need not be simple (a nontrivial cycle is a typical example of a non-backtracking yet non-simple path). The following claim states this observation formally. The proof is straightforward and is included for completeness.

\begin{claim}\label{claim: non-backtracking}
	Let $A$ be the adjacency matrix of a $d$-regular graph $G$. Then, $P_t(A)$ is the $n \times n$ matrix in which the $(u,v)$'th entry equals the number of non-backtracking paths of length exactly $t$ between $u$ and $v$. 
\end{claim}

%That is, the $(u,v)$'th entry in $P_t(A)$ equals the number of non-backtracking paths of legnth $t$ between $u$ and $v$.\\

\begin{proof}
We use induction on $t$. Note that $P_0(A)=I$, $P_1(A)=A$ and $P_2(A)=A^2-dI$ satisfy the claim. For the induction step, suppose that the claim holds for all Geronimus Polynomials of order strictly less than $t$. Consider the term $A\cdot P_{t-1}(A)$, which corresponds to paths of length $t$ such that the first $t-1$ hops on the path are non-backtracking. Note that reducible paths in this term are those paths that can only be reduced by eliminating their last two arcs and so there must be exactly $(d-1)P_{t-2}(A)$ of them. Being the difference between those quantities, it follows that $P_t(A) = A\cdot P_{t-1}(A) - (d-1)P_{t-2}(A)$ corresponds to the non-backtracking paths.
%\footnote{We point out that there is another approach to tackling non-backtracking paths in a graph, namely, Hashimoto's non-backtracking operator~\cite{hashimoto1989zeta}. This operator has been shown to be useful in a variety of scenarios (e.g., localization and centrality~\cite{martin2014localization}, clustering~\cite{krzakala2013spectral}, and percolation~\cite{karrer2014percolation} in networks). Claim~\ref{claim: non-backtracking} implies of an algebraic coreespondance between Hashimoto's operator and the Chebyshev polynomials (as applied, e.g., at \cite{alon2007non}) and the operation of the Geronimus polynomials over the adjacency matrix of the graph.}
\end{proof}

As a corollary, the entries of $P_t(A)$ are non-negative for all $t\geq 0$. In addition, as $d(d-1)^{t-1}$ is the number of  non-backtracking paths of length $t>0$ starting from every $v\in G$, this quantity equals the sum of entries in every row of $P_t(A)$. Hence, Claim~\ref{claim: non-backtracking} implies that $P_t(A)\textbf{1}_n = d(d-1)^{t-1} \textbf{1}_n$.

Summing over the indices $0\leq t\leq k$, yields

\begin{equation} \label{eq:ev}
\sum_{t=0}^{k} P_t(A)\mathbf{1}_n = \left(1 + \sum_{t=1}^{k} d(d-1)^{t-1}\right) \textbf{1}_n = \mu_{d,k} \cdot \textbf{1}_n.
\end{equation}\\

We are now ready to prove Theorem \ref{eq: eigenvalue bound}:

\begin{proof} \textbf{(of Theorem \ref{eq: eigenvalue bound})} Given Claim~\ref{claim: non-backtracking}, the sum of matrices $\sum_{t=0}^{k} P_t(A)$ corresponds to all non-backtracking paths of length at most $k$. Since $G$ is of diameter $k$, this sum of matrices must consist of strictly positive entries, and can thus be represented as $\sum_{t=0}^{k} P_t(A) = J + M$, 
where $J$ is the all ones matrix and $M$ is non-negative. We now have
$
M\mathbf{1}_n = \left(\sum_{t=0}^{k} P_t(A)-J\right)\mathbf{1}_n = \left(\mu_{d,k} - n\right)\mathbf{1}_n
$,
where the second equality is due to~\eqref{eq:ev}. 

Recall that $A$ is symmetric and thus diagonalizable w.r.t.~an orthogonal basis. Therefore, any eigenvector $v\neq \textbf{1}_n$ must lie in $(span\{\mathbf{1}_n\})^\perp$. Since $J\textbf{1}_n=n\textbf{1}_n$ and $rank(J)=1$, it follows that $Jv=0$. Hence,
$
Mv = \left(\sum_{t=0}^{k} P_t(A)-J\right)v = \sum_{t=0}^{k} P_t(A)v 
=\sum_{t=0}^{k} P_t(\lambda)v.
$	

This implies in particular that 
$spec(M) = \left\{\sum_{t=0}^{k} P_t(\lambda)\> |\> \> \lambda \text{ is a nontrivial eigenvalue of } A \right\}\cup\{\mu_{d,k} - n\}$.
We now apply the Perron-Frobenius Theorem (see~\cite{meyer2000matrix}), which states that a non-negative matrix admits a non-negative eigenvector with a non-negative eigenvalue that is larger or equal, in absolute value, to all other eigenvalues. Now, since $\textbf{1}_n$ is the only non-negative eigenvector of $M$, we conclude that $\mu_{d,k} - n$ is the leading eigenvalue of $M$ and the claim follows.
\end{proof}

%\noindent{\bf Proof of Theorem \ref{spectral main}.} 
\subsection{Proof of Theorem \ref{spectral main}}

Our proof of Theorem \ref{spectral main} utilizes a careful asymptotic estimation of the Geronimus Polynomials' coefficients. When $\lambda(G)$ is of order larger than $\sqrt{d}$, our analysis asserts that $\left|\sum_{t=0}^{k} P_t(\lambda)\right|$ must be larger then $O(d^{k/2})$ for some nontrivial eigenvalue $\lambda$ of $G$, thus resulting in a contradiction to Theorem~\ref{eq: eigenvalue bound}.

For our purposes, it will be beneficial to use the representation $P_t(x) = \sum_{i=0}^{t} a_{t,i} x^i$, where $a_{t,i}$ is the $i$'th coefficient of the $t$'th Geronimous Polynomial. We note the following: (i) $P_t$ is either odd or even\footnote{A polynomial $q(x)$ is said to be even if $q(x)=q(-x)$ and odd if $q(-x)=-q(x)$.} for all $t>0$, and the parity of $P_t$ equals the parity of $t$. This can be shown either by induction using the recurrence relation, or straightforward from the solution~\eqref{eq:geronimus}; (ii) A comparison of the leading coefficients in the recurrence implies that $a_{t,t} = a_{t-1,t-1}$. Applying the boundary conditions $(a_{1,1}=a_{0,0}=1)$ yields $a_{t,t} =1$ for all $t>0$; (iii) Setting $\theta=\frac{\pi}{2}$ in~\eqref{eq:geronimus} yields $a_{t,0} = d(d-1)^{t/2-1}(-1)^{t/2}$ whenever $t$ is even.

The following easy-to-prove claim provides us with asymptotic estimates for the rest of the coefficients.  Note that the $\Theta(\cdot)$ notation is hiding factors of $t$ (we will only use this claim only where $t$ is constant). 

\begin{claim}\label{claim: Geronimus coefficients}
	Let $P_t(x) = \sum_{i=0}^{t} a_{t,i} x^i$ denote the Geronimous polynomial of order $t$, then 
	\[ 
	a_{t,i} = \left\{\begin{matrix}
	(-1)^{\frac{t-i}{2}} \Theta\left(d^{\frac{t-i}{2}}\right) & \mbox{if } (t-i) \; \text{is even}\\
	0 & \mbox{if } (t-i) \; \text{is odd}  
	\end{matrix}\right.
	\]
	for all $0\leq i\leq t$.
\end{claim}

%\mdnote{Maybe add somewhere that we assume that $t$ is a constant?  Since in both Gal's original proof and my new version the $\Theta$ is hiding factors of $t$ that I don't know exactly how to analyze.  Of course, that makes ``induction" a little odd.}

%	The parity of $P_t$ implies that $a_{t,i} = 0$ whenever $(t-i)$ is odd. Hence, it suffices to consider the case of an even difference, which we prove by induction on $(t-i)$. As $a_{t,t} = 1$, the basis of the induction holds.
%	
%	Assume that the claim holds for every $t',i'$ for which $t'-i'<t-i$. A comparison of the $i$'th coefficient in the recurrence yields 
%    \[a_{t,i} = a_{t-1,i-1} + (d-1) a_{t-2,i}= a_{t-1,i-1} + (d-1) \Theta(d^{\frac{t-2-i}{2}}) = a_{t-1,i-1} + \Theta(d^{\frac{t-i}{2}})\]
%	where the second equality is due to the induction hypothesis.
%	Now, if $a_{t-1,i-1} = O(d^{\frac{t-i}{2}})$ then the claim holds. If this is not the case, we have
%	$a_{t,i} = \Theta(a_{t-1,i-1})$.
	
%	We continue iteratively and obtain that $a_{t,i} = \Theta (a_{t-i,0}) = \Theta(d^{\frac{t-i}{2}})$,
%	thus proving the claim.

The proof of Claim \ref{claim: Geronimus coefficients} is provided in Appendix \ref{app: alg proofs}. This immediately gives the next corollary, which is just a slightly easier to use formulation of $P_t(x)$.

\begin{corollary} \label{cor:Geronimus}
The Geronimus Polynomial of order $t$ can be written as
\[
P_t(x) = \sum_{i=0}^{\lfloor \frac{t}{2}\rfloor} (-1)^i \cdot \Theta (d^i) \cdot x^{t-2i}.
\]    
\end{corollary}

We are now ready to apply this machinery. The following lemma bounds the value of these polynomials on values which are ``small".

\begin{claim} \label{cl:ev-large}
	Let $\frac{1}{2} < \alpha \leq 1$, and let $|\lambda| = \Theta(d^{\alpha})$.  Then $|P_t(\lambda)| = \Theta(d^{t\alpha})$.
\end{claim}
\begin{proof}
	We use induction on $t$. For $t=0$ we have that $P_0(\lambda) = 1 = \Theta(d^{0\alpha})$, and for $t=1$ we have that $|P_1(\lambda)| = |\lambda| = |\Theta(d^{1\alpha})|$. 
%	P_0(A)v &=Iv =v= \Theta(d^{0})v \\
%	P_1(A)v &=Av =\lambda v= \Theta(d^{\alpha})v
%    \end{align*}
    	Assume that the claim holds for the Geronimus Polynomials of order less than $t$. Using Corollary~\ref{cor:Geronimus}, we now have
    	\begin{align*}
    	    P_t(\lambda) &= \sum_{i=0}^{\lfloor \frac{t}{2}\rfloor} (-1)^i \cdot \Theta (d^i) \cdot \lambda^{t-2i} = \sum_{i=0}^{\lfloor \frac{t}{2}\rfloor} (-1)^i \cdot \Theta (d^i) \cdot \Theta(d^{\alpha(t-2i)}) = \sum_{i=0}^{\lfloor \frac{t}{2}\rfloor} (-1)^i \cdot \Theta (d^{t\alpha + i(1-2\alpha)}). 
    	\end{align*}
 
	Whenever $\alpha > \frac{1}{2}$, the absolute value of this equals $\Theta(d^{t\alpha})$ as claimed. 
\end{proof}

\iffalse
\begin{claim} \label{cl:ev-large}
	Let $v$ be an eigenvector of $A$ with $Av=\lambda v$. If $\lambda = \Theta(d^\alpha)$ for some $\alpha > \frac{1}{2}$, then $P_t(A)v=\Theta(\lambda^t)v = \Theta(d^{t\alpha})v.$
\end{claim}
\begin{proof}
	Again, we use induction on $t$. For $t=0$ we have that $P_0(A)v =Iv =v= \Theta(d^{0})v$, and for $t=1$ we have that $P_1(A)v =Av =\lambda v= \Theta(d^{\alpha})v$.  
%	\begin{align*}
%	P_0(A)v &=Iv =v= \Theta(d^{0})v \\
%	P_1(A)v &=Av =\lambda v= \Theta(d^{\alpha})v
%    \end{align*}
    	Assume that the claim holds for the Geronimus Polynomials of order less than $t$. Using Corollary~\ref{cor:Geronimus}, we now have
\[
	P_t(A)v = P_t(\lambda)v =  
	\sum_{i=0}^{\lfloor \frac{t}{2}\rfloor}\Theta (d^i) \cdot \lambda^{t-2i} v = \sum_{i=0}^{\lfloor \frac{t}{2}\rfloor}\Theta (d^i) \cdot \Theta(d^{\alpha(t-2i)})v = \sum_{i=0}^{\lfloor \frac{t}{2}\rfloor} \Theta(d^{t\alpha + i(1-2\alpha) })v.
\]
	Whenever $\alpha > \frac{1}{2}$ this equals $\Theta(d^{t\alpha})v$. 
\end{proof}

\mdnote{Above theorem/proof wrong because of $\pm$, but I'm not sure how to directly state the fix}

\mdnote{Also, why do we care about $A$ and $v$ in the above?  Can't we just let $\lambda$ be an eigenvalue and analyze $P_t(\lambda)$?  That should make the statement easier, since then don't care about direction v}
\fi

The proof of the theorem follows directly.

\begin{proof}\textbf{(of Theorem~\ref{spectral main})} Suppose that $A$ obtains an eigenvalue  $\lambda = \Theta(d^\alpha)$ for some $\alpha > \frac{1}{2}$. Then, applying Claim \ref{cl:ev-large}, we have:
\begin{align*}
\left|\sum_{t=0}^{k} P_t(\lambda)\right| &\geq |P_k(\lambda)| - \sum_{t=0}^{k-1} |P_t(\lambda)| =  \Theta(d^{k\alpha}) - \sum_{t=0}^{k-1} \Theta(d^{t\alpha})= \Theta(d^{k\alpha})
\end{align*}

This expression, however, is upper bounded by $\mu_{d,k} - n$ (by Theorem \ref{eq: eigenvalue bound}), which is $O(d^{k/2})$ by the assumption of the Theorem~\ref{spectral main}. We thus have
$\Theta(d^{k \alpha}) \leq \mu_{d,k} - n \leq O(d^{k/2})$, 
and this is, of course, a contradiction to the assumption $\alpha > \frac{1}{2}$. We therefore conclude that $\lambda=O(\sqrt{d})$.	
\end{proof}

\iffalse
\begin{proof}\textbf{(of Theorem~\ref{spectral main})} Suppose that $A$ obtains an eigenvalue  $\lambda = \Theta(d^\alpha)$ for some $\alpha > \frac{1}{2}$. Then, applying Claim \ref{cl:ev-large}, we have:
\begin{align*}
\left|\sum_{t=0}^{k} P_t(\lambda)\right| =  \left|\sum_{t=0}^{k} \Theta(d^{t\alpha})v\right| = \Theta(d^{k\alpha})v 
\end{align*}

This expression, however, is upper bounded by $\mu_{d,k} - n$ (by Theorem \ref{eq: eigenvalue bound}), which is $O(d^{k/2})$ by the assumption of the Theorem~\ref{spectral main}. We thus have
\[|\Theta(d^{k \alpha})| \leq \mu_{d,k} - n \leq O(d^{k/2})\]
and this is, of course, a contradiction to the assumption $\alpha > \frac{1}{2}$. We therefore conclude that $\lambda=O(\sqrt{d})$.	
\end{proof}
\fi

\subsection{Proof of Theorem~\ref{weak-expansion}} The proof of Theorem~\ref{weak-expansion} relies on some of the ideas introduced in the proof of Theorem~\ref{spectral main}.  Let $\lambda$ be a nontrivial eigenvalue of $G$.  We wish to show that $|\lambda| \leq O(\eps^{1/k}) d$. If $|\lambda| \leq O(d^{2/3})$ then we are done. Suppose, then, that $|\lambda| \geq \omega(d^{2/3})$, and hence $d = o(|\lambda|^{3/2})$.  Consider the sum $| \sum_{t=0}^k P_t(\lambda)|$.  Corollary~\ref{cor:Geronimus}, and the discussion which showed that $a_{t,t} = 1$, imply that this sum is at least
\begin{align*}
\left| \sum_{t=0}^k P_t(\lambda) \right| &\geq |\lambda^k + \lambda^{k-1}| - \sum_{i=1}^{\lfloor k/2 \rfloor} \left(\Theta(d^i) |\lambda|^{k-2i} + \Theta(d^i) |\lambda|^{k-2i-1}\right) \\
& \geq |\lambda^k + \lambda^{k-1}| - \sum_{i=1}^{\lfloor k/2 \rfloor} \left(\Theta(|\lambda|^{k-(i/2)}) + \Theta(|\lambda|^{k-1 - (i/2)})\right) \geq \Theta(|\lambda^k|),
\end{align*}
where the second inequality follows from the assumption that $|\lambda| \geq \omega(d^{2/3})$. 

When we plug this into Theorem~\ref{eq: eigenvalue bound}, we get that $\Theta(|\lambda^k|) \leq \mu_{d,k} - n \leq \eps \mu_{d,k}$.  Since $\mu_{d,k} \leq c d^k$ for some constant $c$, we get that $c' |\lambda^k| \leq \eps c d^k$ for some constants $c$ and $c'$, and hence 
$|\lambda| \leq \left(\frac{c}{c'}\right)^{1/k} \eps^{1/k} d$,
proving the theorem. \qed

\section{Diameter \emph{vs.}~Combinatorial Expansion} \label{coarse}

We present below our results for combinatorial expansion. We first point out that applying the Cheeger inequality~\cite{hoory2006expander} to our bounds on spectral expansion immediately implies bounds on combinatorial expansion. Specifically, the Cheeger inequality states that $h_e(G) \geq \frac{d-\lambda_2}{2}$. When combined with Theorems~\ref{spectral main} and~\ref{weak-expansion}, this yields the following bounds.

\begin{theorem}
    Let $G$ be a $(d,k)$ graph with $n$ vertices, for some constant $k>0$. If $n \geq \mu_{d,k} - O(d^{k/2})$ then $h_e(G) \geq \frac{d - O(\sqrt{d})}{2}$.
\end{theorem}

\begin{theorem}
    Let $G$ be a $(d,k)$ graph with $n$ vertices, for some constant $k>0$. If $n \geq (1-\eps) \mu_{d,k}$ then $h_e(G) \geq \frac{(1-O(\eps^{1/k}))d}{2}$.
\end{theorem}

Observe that, since clearly $d/2$ is an upper bound on $h_e(G)$, both of these bounds imply very high expansion guarantees when $n$ is very close to the Moore Bound. However, when this is not so, e.g., when $n = \mu_{d,k} / k$, neither bound yields nontrivial expansion guarantees.

To provide stronger expansion guarantees for graphs that do not come very close (additively/ multiplicatively) to the Moore Bound, we analyze combinatorial expansion directly. We next present our bounds for edge and vertex expansion in undirected and directed graphs. We discuss the implications of these expansion bounds for known $(d,k)$-graph constructions in Table~\ref{implications table} and in Appendix~\ref{sec: implications}.

\noindent{\bf Undirected graphs.} Our main result of this section is the following: 

\begin{theorem} \label{coarse ee}
	Let $G=(V,E)$ be a $d$-regular graph of size $n$ and diameter $k$. If $n=\alpha\cdot\mu_{d,k}$, then 
	\[
	h_e(G)\geq\frac{\alpha d}{2k}\cdot\left(1-\frac{1}{(d-1)^k}\right)\quad \text{and} \quad \phi_V(G)\geq \frac{\alpha}{2(k-1) +\alpha}.
	\]
\end{theorem}

Our proof of Theorem~\ref{coarse ee} utilizes a counting argument. As the graph has diameter $k$, each pair of vertices on opposite sides of a cut must be connected via a path of length at most $k$ that traverses the boundary. However, there is an upper bound, induced by the degree and diameter of the graph, on the number of such paths that traverse a given edge/vertex. A careful examination of the implications of these two limitations provides us with a lower bound on the size of the boundary. The proof appears in Appendix~\ref{app:undirected}. 

\noindent{\bf Directed graphs.} We consider directed graphs next. We begin by introducing the relevant terminology and notation. We say that a directed graph (a.k.a.~digraph) $G=(V,E)$ is $d$-regular if both the out-degree and the in-degree of each vertex equals $d$. A cut in a digraph $e(S,S^c)=\{(u,v)\in E\space|u\in S,v\in S^c\}$ is asymmetric, and consists of all edges directed from $S$ to $S^c$. The diameter is still defined as the maximal distance between two vertices, and the corresponding Moore Bound is only slightly different (as there are potentially $d^i$ vertices of distance $i$ from a given vertex):
$\tilde{\mu}_{d,k} = \sum_{i=0}^{k} d^i = \frac{d^{k+1}-1}{d-1}$.

The following result is the directed analogue of Theorem~\ref{coarse ee}, and is proved in Appendix~\ref{app:directed}.

\begin{theorem} \label{coarse digraph}
	Let $G$ be a $d$-regular, $k$-diameter \textbf{directed} graph of size $n=\alpha \cdot \tilde{\mu}_{d,k}$, then
	\[h(G)\geq \frac{\alpha}{2k} \left(d-\frac{1}{d^k}\right) \quad \text{and} \quad \phi_V(G)\geq \frac{\alpha \cdot d}{2(d+1)(k-1)+\alpha \cdot d}.\] 
\end{theorem}

\noindent{\bf Refined results for low-diameter graphs.} Much research on constructing low-diameter graphs focuses on diameters $2$ and $3$ (see, e.g.,~\cite{miller2005moore,erdos1962problem,mckay1998note}). Graphs of very low diameter are particularly important from a practical perspective~\cite{besta2014slim,kim2008technology,kim2009cost}. The following theorems improve upon our results for the edge expansion and vertex expansion of $(d,k)$-graphs.

\begin{theorem} \label{2-diam ee}
	Let $G=(V,E)$ be an undirected $(d,2)$-graph of size $n=\alpha \cdot d^2$. Then 
    $$h_e(G)\geq \frac{2d+1-\sqrt{4(1-\alpha)d^2+4d+1}}{4}.$$
\end{theorem}

\begin{theorem} \label{2-diam ve}
	Let $G=(V,E)$ be an undirected $(d,2)$-graph of size $n= \alpha\cdot d^2$.  Then
	$\phi_V(G)\geq \frac{2\alpha}{2\alpha+1}$.
\end{theorem} 

We can extend our analysis to graphs of diameter $3$, yielding the following theorem.

\begin{theorem} \label{3-diam ve}
	Let $G=(V,E)$ be a $(d,3)$-graph of size $n=\alpha\cdot d^3$, then
	$\phi_V(G)\geq \frac{\alpha}{\alpha+1}$.
\end{theorem}

\section{Conclusion and Open Questions}

We revisited the classical question of relating the expansion and the diameter of graphs and showed that not only do good expanders exhibit low diameter but the converse is also, in some sense, true. We also discussed the implications of our results for constructions from the rich body literature on low-diameter graphs. We leave the reader with many interesting open questions, including: (1) {\bf Tightening the gaps.} An obvious open question is improving upon our lower bounds and establishing upper bounds on the expansion of fixed-diameter graph constructions. %In fact, our only lower bound guaranteed to be essentially tight is our result for the edge expansion of diameter-$2$ graphs (Theorem~\ref{2-diam ee}). 
%§(2) {\bf Can our framework be leveraged to prove impossibility results for large $(d,k)$-graphs?} A possible direction for establishing the nonexistence of large $(d,k)$-graphs might be to show that they have to be impossibly good expanders (for example, their spectral gap may be so large so it would violate lower bounds for $\lambda(G)$). 
(2) {\bf Benchmarking against the optimal (largest possible) $(d,k)$-graph.} We used the Moore Bound as a benchmark. Another approach would be to compare against the size of the largest possible $(d,k)$-graph. (3) {\bf Geronimus Polynomials vs.\ Hashimoto's non-backtracking operator.} The operation of the Geronimus Polynomials over the adjacency matrix of the graph offers a new perspective on its non-backtracking paths (as established in Lemma \ref{claim: non-backtracking}). This suggests a non-trivial relation between these polynomials and Hashimoto's non-backtracking operator, which we believe is of independent interest and may find wider applicability.

\subsubsection*{Acknowledgements}
	
We wish to thank Nati Linial, Alex Samorodnitsky, Elchanan Mossel, and Noga Alon for fruitful discussions.

\newpage
\bibliographystyle{plain}
\bibliography{diam}
	
\newpage 
\section*{Appendix}

\appendix

\section{Proofs for Algebraic Expansion}\label{app: alg proofs}

\begin{proof}(of Claim \ref{claim: Geronimus coefficients})
	The parity of $P_t$ implies that $a_{t,i} = 0$ whenever $t-i$ is odd. Hence, it suffices to consider the case of an even difference, which we prove by induction on pairs $(t,i)$ with $i \leq t$.  Note that $a_{t,t} = 1$ for all $t$ as claimed, and moreover that the claim holds for $(t,i)$ where $t \leq 2$ by construction.  This forms the base case of our induction. 
	
	Now consider some $(t,i)$ with $t \geq 3$ and $i < t$.  It is easy to see from the recurrence relation defining $P_t$ that  $a_{t,i} = a_{t-1,i-1} - (d-1) a_{t-2,i}$, and so by induction we get that 
	\begin{align*}
	    a_{t,i} &= a_{t-1,i-1} - (d-1) a_{t-2, i} \\
	    &= (-1)^{\frac{t-1 - (i-1)}{2}} \Theta(d^{\frac{t-1 - (i-1)}{2}}) - (d-1)((-1)^{\frac{t-2-i}{2}} \Theta(d^{\frac{t-2-i}{2}})) \\
	    & = (-1)^{\frac{t-i}{2}} \Theta(d^{\frac{t-i}{2}}) - (-1)^{\frac{t-2-i}{2}} \Theta(d^{\frac{t-i}{2}}) \\
	    &= (-1)^{\frac{t-i}{2}} \Theta(d^{\frac{t-i}{2}}) + (-1)^{\frac{t-i}{2}} \Theta(d^{\frac{t-i}{2}}) \\
	    &= (-1)^{\frac{t-i}{2}} \Theta(d^{\frac{t-i}{2}})
	\end{align*}
	as claimed.
\end{proof}

\section{Proofs for Combinatorial Expansion}

\subsection{Undirected Graphs} \label{app:undirected}

\begin{proof} (of Theorem \ref{coarse ee}) We use a counting argument. Let $(S,S^c)$ be a cut in the graph, and let $|S|=s\leq \frac{n}{2}$. As the diameter equals $k$, every pair of vertices that lie on both sides of the cut must be connected via a path of length at most $k$. We thus have $s(n-s)$ such paths, each of which passes through some edge in the cut.
	
	How many paths of length at most $k$ include a given edge $e\in E$? As $G$ is $d$-regular, there are at most $(d-1)^{l-1}$ paths of length $l$ for which $e$ is in the $i$'th position in the path. It follows that no more than $l\cdot (d-1)^{l-1}$ paths of length $l$ use a specific edge, hence the number of paths of length at most $k$ that utilize a fixed edge is upper bounded by
	\[
	f_{d-1}(k)=\sum_{l=1}^k l\cdot (d-1)^{l-1}.
	\]
	
	Let us find a simpler formulation of $f_{d-1}(k)$. Integrating yields 
	\[
	F_{d-1}(k)=\sum_{l=1}^k (d-1)^{l}=\frac{(d-1)^{k+1}-1}{(d-1)-1}.
	\]

	Differentiating brings us back to 
	\begin{align*}
  	f_{d-1}(k) &=\frac{(k+1)(d-1)^k(d-2)-[(d-1)^{k+1}-1]}{(d-2)^2} \\
	&\leq \frac{(k+1)(d-1)^k(d-2)-(d-1)^{k}(d-2)}{(d-2)^2} \\
	&= \frac{k(d-1)^{k}}{(d-2)}
    \end{align*}
    
	Now, $s(n-s)$ paths use the cut, and every edge in the cut can be a part of at most $f_{d-1}(k)$ paths. It follows that $|e(S,S^c)|\geq \frac{s(n-s)}{f_{d-1}(k)}$ for every cut $(S,S^c)$ in $G$. Hence, the cut that realizes $h_e(G)$ satisfies
	
	\begin{align*}
	h_e(G) &= \frac{|e(S,S^c)|}{|S|} 
	\geq \frac{s(n-s)(d-2)}{s\cdot k(d-1)^{k} }
	\geq \frac{n}{2}\cdot \frac{(d-2)}{k(d-1)^{k} } \\
	&=\frac{\alpha d \cdot ((d-1)^k-1)}{2(d-2)}\cdot \frac{(d-2)}{k(d-1)^{k}} \\
	&= \frac{\alpha d}{2k} \cdot \left(1-\frac{1}{(d-1)^k}\right). \qedhere
    \end{align*}
	
	A similar argument applies for the vertex expansion.  Let $S\subset V$ be a subset of size $s\leq \frac{n}{2}$, and let $x=|N(S)|$ denote the size of its outer boundary. Then $|S|\cdot |S^c\setminus N(S)|=s(n-s-x)$
	pairs of vertices are connected via a path of length $2\leq l\leq k$ in which one of the inner vertices of the path is in $N(S)$.  But how many paths of this form can there be?
	
	A path of length $l$ consists of $(l-1)$ possible positions for an inner vertex. As there are $d(d-1)^{l-1}$ paths of length $l$ passing through a vertex in a fixed position, we conclude that the number of paths is at most 
	\begin{align*}
	\sum_{l=2}^{k} x(l-1)\cdot d(d-1)^{l-1} &= xd\sum_{l=2}^{k} \left(l(d-1)^{l-1}-(d-1)^{l-1}\right) \\
	&= xd \left(f_{d-1}(k)-1-\frac{(d-1)^k-(d-1)}{d-2}\right) \\
	&\leq xd\left(\frac{k(d-1)^{k}}{(d-2)}-1-\frac{(d-1)^k-(d-1)}{d-2}\right) \\
	&\leq \frac{xd}{d-2}\left((k-1)(d-1)^k+1\right) \\
	&\leq \left(\frac{d+1}{d-2}\right)\cdot x(k-1)(d-1)^k.
	\end{align*}
	
	As the number of paths must exceed the number of pairs connected by a path, it follows that 
	\[
	s(n-s-x)\leq \left(\frac{d+1}{d-2}\right)\cdot x(k-1)(d-1)^k,
	\]
	and thus, the cut that realizes $\phi_V(G)$ satisfies
	
	\begin{align*}
	\phi_V(G) &\geq \frac{x}{s} \geq 
	\frac{n-s}{(\frac{d+1}{d-2}) (k-1)(d-1)^k+s} \\
	&\geq \frac{0.5n}{(\frac{d+1}{d-2}) (k-1)(d-1)^k+0.5n} \\
	&\geq \frac{\alpha \cdot \mu_{d,k}}{2(\frac{d+1}{d-2})(k-1)(d-1)^k +\alpha \cdot \mu_{d,k}} \\
	&\geq \frac{\alpha d\cdot \frac{(d-1)^k-1}{d-2}}{2(\frac{d+1}{d-2})(k-1)(d-1)^k +\alpha \cdot d\cdot \frac{(d-1)^k-1}{d-2}} \\
	&\geq \frac{\alpha}{2(k-1) +\alpha}. \qedhere
	\end{align*}
\end{proof}

\begin{proof} (of Theorem \ref{2-diam ee}) Let $(S,S^c)$ be a cut in the graph with $|S|=s\leq \frac{n}{2}$. Let $x=|e(S,S^c)|$ denote the number of edges in the cut. There are $s(n-s)$ pairs of vertices such that one vertex is in $S$ and the other is in $S^c$, and exactly $x$ of these pairs are connected by a path of length $1$. Hence $s(n-s)-x$ such pairs are connected by a path of length $2$.
	
	How many such length-$2$ paths are there? For each vertex $v\in V$, let $a_v$ denote the number of edges in the cut that are incident on $v$. Given an edge $\{u,v\} \in e(S,S^c)$, the number of length-$2$ paths that use this edge with both ends in different sides of the cut is $(d-a_u)+(d-a_v)$. It follows that the total number of length-$2$ paths with both ends in different sides of the cut is exactly
	\[
	\sum_{\{u,v\} \in e(S,S^c)} ((d-a_u)+(d-a_v))
	= 2dx -	\sum_{\{u,v\}\in e(S,S^c)} (a_u+a_v).
	\]
	
	As every vertex $u\in V$ contributes exactly $a_u$ summands to the sum, the expression above equals
	\[
	2dx - \left(\sum_{u\in S} a_u ^2 + \sum_{u\in S^c} a_u ^2\right).
	\]
	
	Since $\sum_{u\in S} a_u =\sum_{v\in S^c} a_v=x$, this expression is maximized (by employing, e.g., Cauchy-Schwarz inequality) whenever $a_u=\frac{x}{s}$ for all $u\in S$ and $a_v=\frac{x}{n-s}$ for all $v\in S^c$. That is, it is maximized when the edges in the cut are spread evenly between all vertices from the every side of the cut. It follows that the total number of length-$2$ paths which cross the cut is at most 
	\[
	2dx-s\left(\frac{x}{s}\right)^2-(n-s)\left(\frac{x}{n-s}\right)^2
	= x\left(2d-\frac{x}{s}-\frac{x}{n-s}\right).
	\]
	
	This number of paths should connect $s(n-s)-x$ pairs of vertices from both sides of the cut. We thus get that 
    $s(n-s)-x\leq x\left(2d-\frac{x}{s}-\frac{x}{n-s}\right).$
	Rearranging terms yields
	$x^2\cdot \frac{n}{s(n-s)} -x(2d+1)+s(n-s)\leq 0.$
	
	It follows that
	\[
	\frac{2d+1-\sqrt{(2d+1)^2-4n}}{2\frac{n}{s(n-s)}}\leq x \leq \frac{2d+1+\sqrt{(2d+1)^2-4n}}{2\frac{n}{s(n-s)}}.
	\]
	
	This means that the size of the cut is bounded from below and from above. In order to lower bound the edge expansion, we only need to use the inequality on the left:
	\[
	h_e(G) = \frac{|e(S,S^c)|}{|S|} = \frac{x}{s}
	\geq \frac{2d+1-\sqrt{(2d+1)^2-4n}}{2\left(\frac{n}{n-s}\right)} \\
	\geq \frac{2d+1-\sqrt{(2d+1)^2-4\alpha\cdot d^2}}{4}. \qedhere
	\]
\end{proof}

\subsection{Directed Graphs} \label{app:directed}

\begin{proof} (of Theorem \ref{coarse digraph}) We follow the steps of the analysis in the undirected case, starting with edge expansion. As before, let $(S, S^c)$ be the cut that realizes $h_e(G)$ with $|S| = s \leq n/2$.  The number of paths of length at most $k$ that utilize a specific edge is now bounded by
	\[f_d(k) = \sum_{l=1}^k l\cdot d^{l-1} \leq \frac{k\cdot d^{k}}{(d-1)},\]
	and since $|e(S,S^c)|\geq \frac{s(n-s)}{f_{d}(k)}$, we have
	\begin{align*}
	h_e(G) &= \frac{|e(S,S^c)|}{|S|} 
	\geq \frac{s(n-s)(d-1)}{s\cdot k\cdot d^{k} }
	\geq \frac{n}{2}\cdot \frac{(d-1)}{kd^{k} }
	= \frac{\alpha \cdot \mu_{d,k}(d-1)}{2k\cdot d^k} \\
	&= \frac{\alpha \cdot (d^{k+1}-1)}{2k\cdot d^k}.
	\end{align*}
	
	For the vertex expansion, we again consider a cut $(S, S^c)$ with $|S| = s \leq n/2$ and $x = |N(S)|$.  In a digraph there are $d^l$ paths of length $l$ passing through a vertex in a fixed position. It follows that the number of paths of length $2\leq l\leq k$ which include a vertex of $N(S)$ as one of its inner vertices is at most
	\begin{align*}
  	\sum_{l=2}^{k} x(l-1)\cdot d^l &= xd\sum_{l=2}^{k} (ld^{l-1}-d^{l-1}) \\
	&= xd\left(f_{d}(k)-1-\frac{d^k-d}{d-1}\right) \\
	&\leq xd\left(\frac{kd^{k}}{(d-1)}-1-\frac{d^k-d}{d-1}\right) \\
    &\leq \left(\frac{d+1}{d-1}\right)\cdot x(k-1)d^k. 
	\end{align*}
	
	It follows that 
	\[s(n-s-x)\leq \left(\frac{d+1}{d-1}\right)\cdot x(k-1)d^k\]
	
	Hence
	\begin{align*}
	\phi_V(G) &= \frac{x}{s} \geq 
	\frac{n-s}{\left(\frac{d+1}{d-1}\right) (k-1)d^k+s} \\
	&\geq \frac{0.5n}{\left(\frac{d+1}{d-1}\right) (k-1)d^k+0.5n} \\
	&\geq \frac{\alpha \cdot \mu_{d,k}}{2\left(\frac{d+1}{d-1}\right)(k-1)d^k+\alpha \cdot \mu_{d,k}} \\
	&= \frac{\alpha \cdot d^{k+1}}{2(d+1)(k-1)d^k+\alpha \cdot d^{k+1}} \\
	&= \frac{\alpha \cdot d}{2(d+1)(k-1)+\alpha \cdot d},
	\end{align*}
	as claimed.
\end{proof}

\subsection{Diameter-$2$ and Diameter-$3$ Graphs} \label{app:refined}

\begin{proof}(of Theorem \ref{2-diam ve}) Let $S\subset V$ be a subset of size $s\leq \frac{n}{2}$, and let $x=|N(S)|$ denote the size of its outer boundary. Then at least
	$|S|\cdot |S^c\setminus N(S)|=s(n-s-x)$
	pairs of vertices are connected via length-$2$ paths whose middle vertex lies in $N(S)$. However, the number of such paths through a fixed middle vertex is at most $\left(\frac{d}{2}\right)^2$. It follows that
	\[s(n-s-x)\leq x\left(\frac{d}{2}\right)^2,\]
	and rearranging this yields the inequality
	\[s(n-s) \leq x\left(\frac{d^2}{4}+s\right).\]
	We thus have
	\[\phi_V(G) \geq \frac{x}{s} \geq \frac{n-s}{0.25d^2+s}
	\geq \frac{0.5n}{0.25d^2+0.5n} = \frac{2\alpha}{2\alpha+1}. \qedhere\]
\end{proof}

\begin{proof}(of theorem \ref{3-diam ve}) Again, let $S\subset V$ be a subset of size $s\leq \frac{n}{2}$, and let $x=|N(S)|$ denote the size of its outer boundary. Then at least
	$|S|\cdot |S^c\setminus N(S)|=s(n-s-x)$
	pairs of vertices are connected via paths of length $2$ or $3$ which pass through $N(S)$. Note that a length-$3$ path of this kind must contain a length-$2$ path whose middle vertex lies in $N(S)$ and both ends in different sides of the cut. It follows that there are at most $2\left(\frac{d}{2}\right)^2(d-1)$ such length-$3$ paths and $\left(\frac{d}{2}\right)^2$ such length-$2$ paths. This implies that
	\[
	s(n-s-x)\leq x\cdot \left(2\left(\frac{d}{2}\right)^2(d-1)+\left(\frac{d}{2}\right)^2\right),
	\]
	and hence
	\[s(n-s)\leq x\cdot \left(\frac{d^2}{2}(d-1)+\frac{d^2}{4}+s\right).\]
	We thus have
	\[
	\phi_V(G) = \frac{x}{s} \geq \frac{n-s}{\frac{d^2}{2}(d-1)+\frac{d^2}{4}+s}
	\geq \frac{0.5n}{\frac{d^3}{2} - \frac{d^2}{4}+ 0.5n} = \frac{\alpha}{\alpha+1-\frac{1}{d}}. \qedhere
	\]
\end{proof}

\section{Implications for Known Constructions} \label{sec: implications}

\subsection{Undirected Graphs}
The results in section \ref{coarse} can be directly applied to obtain expansion guarantees for well known constructions of large $(d,k)$-graphs. The most general of these is perhaps the \textit{undirected de Bruijn graph} \cite{de1946combinatorial}, which may be constructed for every diameter $k$ and even degree $d$ (the detailed definition may be found in~\cite{miller2005moore}). These graphs are of size $n \geq \left(\frac{d}{2}\right)^k$ and have been extensively applied in various contexts, including the design of feedback registers \cite{fredricksen1982survey,lempel1970homomorphism}, decoders \cite{collins1992vlsi}, and computer networks \cite{bermond1989large,bermond1989bruijn,esfahanian1985fault,leighton2014introduction,pradhan1985fault}. 

As for the spectral guarantees for this construction, since the second eigenvalue of these graphs is known to be $\lambda_2=d\cos\left(\frac{\pi}{k+1}\right)$ ~\cite{delorme1998spectrum}, applying the Cheeger inequality yields 
\begin{align*}
h_e(G) &\geq \frac{d-d\cos(\frac{\pi}{k+1})}{2} \sim \frac{d}{4} \left(\frac{\pi}{k+1}\right)^2.
\end{align*}
The vertex expansion is thus
\begin{align*}
\phi_V(G) \geq \frac{1}{4} \left(\frac{\pi}{k+1}\right)^2.
\end{align*}
While applying Theorem~\ref{coarse ee} implies weaker guarantees, whenever $k=2$ this construction is a $\frac{1}{4}$-approximation to the Moore Bound, and thus by the refined argument in Theorem \ref{2-diam ve}, we have
\[\phi_V(G)\geq \frac{1}{3}.\]
This constitutes the best-known vertex-expansion guarantee for $(d,2)$-de Bruijn graphs. 

Canale and Gomez \cite{canale2005asymptotically} made considerable progress on the degree-diameter problem by giving a construction of $(d,k)$-graphs of size $n \geq \left(\frac{d}{1.57}\right)^k$ for an infinite set of values $d$ and $k$. In these graphs the expansion guarantees from our theorems are slightly better:
\begin{align*}
h_e(G) &\geq \frac{d}{2k\cdot 1.57^k} \cdot \left(1-\frac{1}{(d-1)^k}\right), \text{ and} \\
\phi_V(G) &\geq \frac{1.57^{-k}}{2(k-1) +1.57^{-k}}.
\end{align*}

This, to the best of our knowledge, is the first analysis of the Canale-Gomez graphs and thus these results constitute the highest expansion guarantees for this construction.

Since the only known constructions that actually draw close to the Moore Bound are of diameter $k=2$, this case is of particular importance for us. The largest known such constructions are based on \textit{polarity graphs}, first introduced by Erd\H{o}s and Renyi~\cite{erdos1962problem} and then independently by Brown~\cite{brown1966graphs}. The design of these graphs makes use of finite  projective geometries in order to produce $d$-regular graphs of diameter $2$ and of size $n=d^2 - d+1$. Another important construction, that attempts to yield large $(d,k)$-graphs that (unlike polarity graphs) are also \textit{vertex-transitive}, was introduced by McKay, Miller and Siran in \cite{mckay1998note}. This property aims to capture some sort of symmetry by the requirement that the automorphism group of the graph acts transitively upon its vertices. This construction, known as MMS-graphs, is of size  $n=\frac{8}{9}(d+\frac{1}{2})^2$ and diameter $2$ and was proposed as the topology of high performance computing networks in \cite{besta2014slim} due to its good performance in simulation in terms of latency, bandwidth, resiliency, cost, and power consumption.

Applying Theorems~\ref{2-diam ee} and \ref{2-diam ve} to \textit{polarity graphs} imply that these graphs enjoy expansion of
\begin{align*}
h_e(G) &\geq \frac{2d+1-\sqrt{4d+1}}{4}, \text{ and } \\
\phi_V(G) &\geq \frac{2}{3}.
\end{align*}

We note that as $\frac{d}{2}-\frac{\sqrt{d-1}}{2}$ and $\frac{1}{2}$ are the best known lower bound for the edge and vertex expansion respectively (obtained by applying the Cheeger inequality for the known spectral gap of these graphs), both bounds depicted here constitute the best expansion guarantees to date for this important construction.

Applying the same theorems for \textit{MMS-graphs} yields
\begin{align*}
h_e(G) &\geq \frac{2d+1-\sqrt{\frac{4}{9}d^2+4d+1}}{4} \approx \frac{d}{3}, \text{ and} \\
\phi_V(G) &\geq \frac{\frac{16}{9}}{\frac{16}{9}+1} = \frac{16}{25}.
\end{align*}

Here, the best known bounds are $\frac{2d+1}{6}$ and $\frac{1}{3}+\frac{1}{6d}$ respectively (also derived from the Cheeger inequality). Note that while the edge expansion guarantee presented here is slightly weaker, the vertex expansion guarantee is substantially tighter.

\subsection{Directed Graphs}
In the case of directed graphs, the state of the art for degree $d\geq2$ and diameter $k\geq 4$ are graphs of size $n = 25\cdot 2^{k-4}$ obtained from the \textit{Alegre digraph} (see \cite{miller2005moore}) and its iterated line digraphs. Applying Theorem~\ref{coarse digraph} yields the bounds
\begin{align*}
h(G) &\geq \left(\frac{2}{d}\right)^k \cdot \frac{25}{16} \cdot \frac{1}{2k} \left(d-\frac{1}{d^k}\right), \text{ and} \\
\phi_V(G) &\geq \frac{\left(\frac{2}{d}\right)^k \cdot \frac{25}{16} \cdot d}{2(d+1)(k-1)+\left(\frac{2}{d}\right)^k \cdot \frac{25}{16} \cdot d}.
\end{align*}

For the remaining values of degree and diameter, the iterated line digraphs of complete digraphs (known in the literature as \textit{Kautz digraphs}~\cite{elspas1968theory}) have been proposed as the underlying topology in the design of computer networks and architectures in \cite{bermond1989bruijn,bermond1989large}. These graphs are of size $n = d^k + d^{k-1}$, and thus by Theorem~\ref{coarse digraph} enjoy expansion of 
\begin{align*}
h(G) &\geq \frac{1}{2k} \left(d-\frac{1}{d^k}\right), \text{ and} \\
\phi_V(G) &\geq \frac{d}{2(d+1)(k-1)+d}.
\end{align*}

Here again, these bound represent the best expansion guarantees to date. While these expressions do not demonstrate near-optimal expansion, let us recall that most applications of expander graphs only require $h(G)$ and $\phi _V (G)$ to be bounded away from zero (see~\cite{hoory2006expander} for a variety of examples). Hence these bounds suffice for a number of desired properties and potential applications whenever the diameter is sufficiently low.

\end{document}